\documentclass[11pt]{article}
\usepackage{latexsym,euscript,amsmath,amssymb,amsbsy,amsfonts,amsthm,amsopn,amstext,amsxtra,amscd}
\usepackage{epsfig}
\usepackage{graphics}
\usepackage{url}
\usepackage[T1]{fontenc}
\allowdisplaybreaks
\theoremstyle{plain}
\newtheorem{theorem}{Theorem}

\newtheorem{corollary}{Corollary}

\theoremstyle{definition}

\newtheorem{remark}{Remark}
\makeatletter\def\blfootnote{\xdef\@thefnmark{}\@footnotetext}\makeatother
\def\dd{\text{\rm d}}
\begin{document}

\title{An extremal problem in uniform distribution theory\\
{\small Dedicated to Harald Niederreiter on the occasion of his 70th birthday}}

\author{V. Bal\'a\v z, M.R. Iac\`o, O. Strauch, S. Thonhauser,  R.F. Tichy \thanks{The authors were
supported by the bilateral Austria-Slovakia travel grant \textquotedblleft Uniform distribution, copulas
and applications\textquotedblright. The second, forth and fifth authors
are supported by the Austrian Science Fund (FWF) Project F5510 (part of the Special Research
Program (SFB) \textquotedblleft Quasi-Monte Carlo Methods: Theory and Applications\textquotedblright).
The second author is also partially supported by the Austrian Science Fund (FWF): W1230,
Doctoral Program \textquotedblleft Discrete Mathematics\textquotedblright.
The first and third authors are supported by VEGA Project 2/0149/14.}}
\date{}
\maketitle
\blfootnote{2010 {\it Mathematical Subject Classification}. 11K06, 60E05, 60A10, 49K27.}
\blfootnote{{\bf Keywords}: Uniform distribution, copula, Monge-Kantorovich problem, dual problem, $c$-convex function, $c$-subdifferential.}

\begin{abstract}
In this paper we consider an optimization problem for Ces\`aro means of bivariate functions. We apply
methods from uniform distribution theory, calculus of variations and ideas from the theory of optimal transport.
\end{abstract}
\section{Introduction}
In a series of papers J.G. van der Corput \cite{van_der_Corput_1935, van_der_Corput_1936} systematically
investigated distribution functions of sequences of real numbers. More recently, the study of distribution
functions was extended to multivariate functions by the Slovak school of O. Strauch and his
coworkers; see \cite{Strauch_school_III, Strauch_school_IV, Strauch_school_II, Strauch_school_I}.
These investigations include the study of the set of all distribution functions of a given sequence and various optimization problems.\\
A particularly interesting problem is the study of extremal limits of two-dimensional sequences of the form
\begin{equation}\label{inte}
\frac{1}{N} \sum_{n = 1}^N F(x_n, y_n)\ ,\qquad N=1,2,\dots
\end{equation}
where $(x_n)_{n\in\mathbb{N}}, (y_n)_{n\in\mathbb{N}}$ are uniformly distributed (for short u.d.) sequences in
the unit interval and $F$ is a given continuous function on $[0,1]^2$; see \cite{PS}.\\

Let us recall that a sequence $(x_n)_{n\in\mathbb{N}}$ of points in $[0,1[$ is said to be u.d.\ if and only if
\begin{equation*}
\lim_{N \rightarrow \infty} \frac{1}{N} \sum_{n = 1}^N \mathbf{1}_{[a, b[} (x_n) = b-a
\end{equation*}
for all intervals $[a, b[ \subseteq [0,1[$, where $\mathbf{1}_E$ denotes as usual the indicator function of
the set $E$. We refer to \cite{DrmotaTichy1997, KN, StrauchPorubsky} as general references on the subject.\\
A mapping $T$ of the unit interval into itself is called uniform distribution preserving (for short u.d.p.)
if the sequence $(T(x_n ))_{n\in\mathbb{N}}$ is u.d.\  whenever $(x_n)_{n\in\mathbb{N}}$ is a u.d.\ sequence in $[0,1[$.\\
These maps have been extensively studied (see e.g. \cite{Bosch, PSS1988}), also in connection to variational problems \cite{Steinerberger}
and extended to compact metric spaces \cite{TichyWinkler}. They are particularly interesting for the purposes of this paper since
they can be thought of as suitable measure preserving rearrangements of the unit interval, as we will see in the next section.
\\
It turned out that the study of the asymptotic behaviour of mean values \eqref{inte} is equivalent to find optimal bounds
for Riemann-Stieltjes integrals of the form
\begin{equation}
\label{1}
\int_0^1\int_0^1F(x,y)dC(x,y),
\end{equation}
where $C$ is the asymptotic distribution function of the sequence $(x_n, y_n)_{n\in\mathbb{N}}$ and is usually referred to as
copula (see \cite{FialovaStrauch}). More precisely, a 2-copula is a function $C\colon [0,1]^2\rightarrow [0,1]$ satisfying the
following properties: for every $x,y \in [0,1]$
\begin{align*}
C(x,0) &= C(0,y) = 0,\\
C(x,1) &= x\;\text{and}\; C(1,y) = y,
\end{align*}
and for every $x_1,x_2,y_1,y_2 \in [0,1]$ with $x_2 \geq x_1$ and $y_2 \geq y_1$
\begin{equation*}
 C(x_2,y_2) - C(x_2, y_1) - C(x_1, y_2) + C(x_1,y_1) \geq 0.
\end{equation*}
An important property of copulas, which makes the expression in (\ref{1}) meaningful and which can be derived from
the above properties, is that every copula $C$ induces a doubly stochastic measure $\gamma_C$
(later on denoted by $\gamma$ when there is no possibility of confusion) on the measurable space $([0,1]^2, \mathcal{B})$, via the formula
\begin{equation*}
\gamma([a,b]\times[c,d])=C(b,d)-C(b,c)-C(a,d)+C(a,c).
\end{equation*}
Moreover,
there is a one-to-one correspondence between copulas and doubly stochastic measures. For every copula $C$,
the measure $\gamma$ is doubly stochastic in the sense that for every Borel set
$B\subset [0,1]$,  $\gamma([0, 1] \times B) = \gamma (B \times [0, 1])= \lambda(B)$ where $\lambda$ is the Lebesgue measure on
$[0,1]$. Conversely, for every doubly stochastic measure $\mu$, there exists a copula $C$ given by
$C(u, v)=\mu(([0, u])\times([0, v]))$. Clearly, a probability measure on $([0,1]^2,\mathcal{B})$
with uniform marginals is doubly stochastic. We refer to \cite{Durante_Sempi, Jaworski, Nelsen} for details.\\
With a slight abuse of notation we write $\int_0^1\int_0^1F(x,y)dC(x,y)$ for the integral with respect to the measure $\gamma$.\\

It should be remarked that copulas are very popular in applied probability. In particular, they are used in
financial mathematics for modeling dependency structures among different kinds of risks;
see \cite{Beigl2013, Embrechtsbook, PucRuesch2013, RueschenBook2013}.\\

Notice, however, that in the calculation of these integrals one can also take advantage of the probabilistic
interpretation of a copula, as the joint distribution function of a random vector with uniform marginals.
More precisely, consider a random vector $(X_1,X_2)$ and suppose that its marginals $F_1$ and $F_2$ are continuous.
By applying the probability integral transformation to each
component, the random vector
\begin{equation*}
(U_1,U_2)= (F_1(X_1), F_2(X_2))
\end{equation*}
has uniform marginals. The copula of $(X_1,X_2)$ is then defined as the joint
cumulative distribution function (cdf) of $(U_1,U_2)$:
\begin{equation*}
C(u_1,u_2)=P(U_1\leq u_1, U_2\leq u_2).
\end{equation*}
In this setting, the notion of optimal coupling can be stated. Let $\mu$, $\nu$ be two probability measures
on $(\mathbb{R}, \mathcal{B})$. Coupling $\mu$ and $\nu$ means constructing two random variables $X_1, X_2$
on $\mathbb{R}$ in such a way that $X_1 \stackrel{d}{\sim} \mu$ and $X_2\stackrel{d}{\sim} \nu$, with
$\stackrel{d}{\sim}$ denoting equality in distribution. The couple $(X_1,X_2)$ is called a coupling of $(\mu,\nu)$.
Now, if we introduce a cost function $c(x,y)$ on $\mathbb{R}\times\mathbb{R}$, then the problem of finding
\begin{equation*}
\inf \mathbb{E}(c(X_1,X_2))
\end{equation*}
where the pair $(X_1,X_2)$ runs over all possible couplings of $(\mu,\nu)$ is called $c$-optimal coupling or
Monge-Kantorovich mass transportation problem (see e.g. \cite{Rueschen, Villani}). Equivalently, this problem can be stated in terms of measures,
\begin{equation*}
\inf\int c(x,y)dP(x,y),
\end{equation*}
where the infimum runs  over all joint probability measures $P$ on $\mathbb{R}\times\mathbb{R}$ with marginals $\mu$ and $\nu$.
For connections of extremal limits and copulas to the theory of optimal transport we refer to \cite{ITT}.\\
Through this paper we consider the equivalent $\sup$ problem
\begin{align}\label{eq:coupling}
\sup\{\mathbb{E}(c(X_1,X_2))\,\vert\,X_1,\,X_2\,\mbox{couplings of}\,\mu,\nu\,\mbox{with}\,P_{X_1}=\mu\,\mbox{and}\,P_{X_2}=\nu\}.
\end{align}
Furthermore, we focus our attention on couplings between uniform distributions. In Section 4 we use this approach for
solving a specific instance of maximizing the integral \eqref{1}.\\

A useful criterion for checking the optimality of a candidate solution is based on the notion of $c$-convexity.
A function $f:X\to\mathbb{R}$ is called $c$-convex if it has a representation $f(x)=\sup_y\{c(x,y)+a(y)\},$ for some function $a$.
The associated $c$-subdifferential of $f$ at $x$ is then defined through:
$$\partial_c f(x)=\{y\,\vert\,f(z)-f(x)\geq c(z,y)-c(x,y)\;\forall\,z\in X\}$$
and $\partial_c f=\{(x,y)\in X\times Y\,\vert\,y\in\partial_c f(x)\}$.\\
Notice, see \cite{Rueschen}, that $y\in\partial_c f(x)$
if and only if $\exists\,a(=a(y))\in\mathbb{R}$ such that
\begin{align}\label{eq:ubdiffeasy}
\psi_{y,a}(x)=c(x,y)+a(y)=f(x)\quad\mbox{and}\quad \psi_{y,a}(\xi)=c(\xi,y)+a(y)\leq f(\xi),
\end{align}
for every $\xi\in X$.\\
The dual problem of (\ref{eq:coupling}) is given by:
\begin{align*}
I(c)=\inf\left\{\int h_1\,d\mu+\int h_2 \,d\nu\,\vert\,c\leq h_1+ h_2,\,h_1\in L^1(\mu)\,\mbox{and}\,h_2\in L^1(\nu)\right\},
\end{align*}
and its study is the basis of the following theorem.
\begin{theorem}[Th. 4.7 from \cite{Rueschen}]\label{Th:optimality}
Let $c$ be such that $c(x,y)\geq a(x)+b(y)$ for some $a\in L^1(\mu)$, $b\in L^1(\nu)$) and assume finiteness of $I(c)$.
Then a pair $(X_1,X_2)$ with $X_1\stackrel{d}{\sim}\mu$, $X_2\stackrel{d}{\sim}\nu$ is an optimal $c-$coupling between $\mu$ and $\nu$
if and only if $$(X_1,X_2)\in\partial_c f\;\mbox{a.s.}$$ for some $c$-convex function $f$, equivalently, $X_2\in\partial_c f(X_1)$ a.s.
\end{theorem}
\section{Main results}
As already pointed out in \cite{FialovaStrauch}, the solution of problem \eqref{1}
depends on the sign of the partial derivative $D_2=\frac{\partial^2F(x,y)}{\partial x\partial y}$.
Special cases have been already studied in the literature, like those described in Fig.1 and Fig. 2. (see \cite{FialovaStrauch}).
In particular, the upper and lower bounds for the first case are given precisely by the Fr\'echet-Hoeffding bounds,
while in the second case the authors provide a criterion \cite[Theorem 7]{FialovaStrauch} to find the corresponding extrema.
\begin{center}
\setlength{\unitlength}{1cm}
\begin{picture}(6,4)
\small
\thinlines
\put(-.3,0.4){$(0,0)$}
\put(2,0.4){$(1,0)$}
\put(-.3,3.4){$(0,1)$}
\put(2,3.4){$(1,1)$}
\put(.6,0){Figure 1.}
\put(.4,1.9){$D_2>0$}
\put(0,.75){\line(1,0){2.5}}
\put(2.5,.75){\line(0,1){2.5}}
\put(2.5,3.25){\line(-1,0){2.5}}
\put(0,3.25){\line(0,-1){2.5}}
\put(4,.75){\line(1,0){2.5}}
\put(6.5,.75){\line(0,1){2.5}}
\put(6.5,3.25){\line(-1,0){2.5}}
\put(4,3.25){\line(0,-1){2.5}}
\put(4,2){\line(1,0){2.5}}
\put(6.7,2){$Y$}
\put(4.4,1){$D_2>0$}
\put(4.4,2.3){$D_2<0$}
\put(3.7,0.4){$(0,0)$}
\put(6,0.4){$(1,0)$}
\put(3.7,3.4){$(0,1)$}
\put(6,3.4){$(1,1)$}
\put(4.6,0){Figure 2.}
\end{picture}
\end{center}
\normalsize
In this paper we maximize \eqref{1} in the special situation described in Fig.\ 3 as a problem of optimal coupling
(see \cite{Uckelmann}) and we provide a criterion for the instance of Fig. 4.
\begin{center}
\setlength{\unitlength}{1cm}
\begin{picture}(6,4)
\small
\thinlines
\put(-1.3,0.4){$(0,0)$}
\put(1,0.4){$(1,0)$}
\put(-1.3,3.4){$(0,1)$}
\put(1,3.4){$(1,1)$}
\put(-0.4,0){Figure 3.}
\put(-1,.75){\line(1,0){2.5}}
\put(1.5,.75){\line(0,1){2.5}}
\put(1.5,3.25){\line(-1,0){2.5}}
\put(-1,3.25){\line(0,-1){2.5}}
\put(1.5,0.75){\line(-1,1){2.5}}

\put(4,.75){\line(1,0){2.5}}
\put(4,.75){\line(0,1){2.5}}
\put(6.5,3.25){\line(-1,0){2.5}}
\put(6.5,3.25){\line(0,-1){2.5}}
\put(1.5,0.75){\line(-1,1){2.5}}
\put(4.83,.75){\line(0,1){2.5}}
\put(5.66,.75){\line(0,1){2.5}}

\put(4,2.2){\tiny{$D_2>0$}}
\put(4.84,1.7){\tiny{$D_2<0$}}
\put(5.67,2.2){\tiny{$D_2>0$}}

\put(.02,0.9){$D_2>0$}
\put(.4,2.6){$D_2<0$}
\put(3.7,0.4){$(0,0)$}
\put(6.5,0.4){$(1,0)$}
\put(3.7,3.4){$(0,1)$}
\put(6.5,3.4){$(1,1)$}
\put(5,0){Figure 4.}
\end{picture}
\end{center}
\normalsize
We start by determining the copula which maximizes \eqref{1} when the sign of the second derivative changes as
described in Fig.\ 4.\ We apply the following criterion \cite[Theorem 7]{FialovaStrauch}.
\begin{theorem}
Let us assume that a copula $C(x,y)$ maximizes the integral
$\int_0^1\int_0^1F(x,y)d{\tilde{C}}(x,y)$. 
Let $[X_1,X_2]\times[Y_1,Y_2]$ be an interval in $[0,1]^2$
such that
\begin{equation}
C(X_2,Y_2)+C(X_1,Y_1)-C(X_1,Y_2)-C(X_2,Y_1)>0
\end{equation}
and such that for every interior point $(x,y)$ the mixed second derivative $D_2$ has constant sign. Then we have:
\par{\rm(i)} 
if $D_2>0$, then
\begin{equation}
\label{eq75}
C(x,y)=\min(C(x,Y_2)+C(X_1,y)-C(X_1,Y_2),C(x,Y_1)+C(X_2,y)-C(X_2,Y_1)),
\end{equation}
\par{\rm(ii)}
if $D_2<0$, then
\begin{equation}
\label{eq76}
C(x,y)=\max(C(x,Y_2)+C(X_2,y)-C(X_2,Y_2),C(x,Y_1)+C(X_1,y)-C(X_1,Y_1)),
\end{equation}
\par\noindent
for every $(x,y)\in[X_1,X_2]\times[Y_1,Y_2]$.
\end{theorem}
This result can be illustrated by the following pictures, where the symbols $\oplus$ and $\ominus$ in a corner mean
that the value of $C$ in that point is taken with positive and negative sign, respectively.
\begin{center}
\setlength{\unitlength}{1cm}
\begin{picture}(6,4)

\put(-1,.75){\line(1,0){2.5}}
\put(1.5,.75){\line(0,1){2.5}}
\put(1.5,3.25){\line(-1,0){2.5}}
\put(-1,3.25){\line(0,-1){2.5}}

\put(6,.75){\line(1,0){2.5}}
\put(6,3.25){\line(0,-1){2.5}}
\put(8.5,.75){\line(0,1){2.5}}
\put(8.5,3.25){\line(-1,0){2.5}}
\put(6,.75){$\ominus$}
\put(8.25,3){$\ominus$}
\put(7.05,.75){$\oplus$}
\put(4.8,1.9){$(X_1,y)$}
\put(4.8,0.3){$(X_1,Y_1)$}
\put(8.,0.3){$(X_2,Y_1)$}
\put(4.8,3.4){$(X_1,Y_2)$}
\put(8,3.4){$(X_2,Y_2)$}
\put(6.8,1.5){$(x,y)$}
\put(6.6,-0.2){Figure 6.}
\put(6.7,3.4){$(x,Y_2)$}
\put(6.7,.3){$(x,Y_1)$}
\put(8.5,1.9){$(X_2,y)$}
\put(6,1.9){\line(1,0){2.5}}
\put(7.3,.75){\line(0,1){2.5}}
\put(8.25,1.9){$\oplus$}
\put(7.3,3){$\oplus$}

\put(6,1.65){$\oplus$}

\put(-1,1.9){\line(1,0){2.5}}
\put(0.3,.75){\line(0,1){2.5}}
\put(-1,1.9){$\oplus$}
\put(0.05,3){$\oplus$}
\put(-1,3){$\ominus$}
\put(-2.2,1.9){$(X_1,y)$}
\put(-0.3,3.4){$(x,Y_2)$}
\put(1.25,0.75){$\ominus$}
\put(0.27,.75){$\oplus$}
\put(1.25,1.67){$\oplus$}
\put(-0.3,.3){$(x,Y_1)$}
\put(1.5,1.9){$(X_2,y)$}
\put(-2.2,0.3){$(X_1,Y_1)$}
\put(1.3,0.3){$(X_2,Y_1)$}
\put(-2.2,3.4){$(X_1,Y_2)$}
\put(1.3,3.4){$(X_2,Y_2)$}
\put(-0.2,1.5){$(x,y)$}
\put(-0.4,-0.2){Figure 5.}
\end{picture}
\end{center}
In order to apply this criterion to the case described in Fig.\ 4 we divide the unit square $[0,1]^2$
into $[0,x_1]\times[0,1]$, $[x_1,x_2]\times[0,1]$ and $[x_2,1]\times[0,1]$, as can be seen in Fig. 7.\\
Then, following the above statement, if $x\in(0,x_1)\cup(x_2,1)$,
we apply \eqref{eq75} in the same way as in Fig.\ 5 and if $x\in(x_1,x_2)$, we apply \eqref{eq76}  as in Fig.\ 6.
\begin{center}
\setlength{\unitlength}{0.72cm}
\begin{picture}(12,10)
\footnotesize
\thicklines
\put(0,0){\line(1,0){12}}
\put(12,0){\line(0,1){8}}
\put(8,0){\line(0,1){8}}
\put(12,8){\line(-1,0){12}}
\put(0,8){\line(0,-1){8}}
\put(-1,0){$(0,0)$}
\put(12,0){$(1,0)$}
\put(-1,8){$(0,1)$}
\put(0,7.8){$\ominus$}
\put(12,8){$(1,1)$}
\put(6,-0.7){Figure 7.}
\put(4.2,0){\line(0,1){8}}
\put(1.3,8.7){$D_2>0$}
\put(8.8,8.7){$D_2>0$}
\put(1.45,8.1){$(x,1)$}
\put(3.7,8.1){$(x_1,1)$}
\put(7.5,8.1){$(x_2,1)$}
\put(1.8,7.8){$\oplus$}
\put(1.2,5.1){$(x,y)$}
\put(0,5){$\oplus$}
\put(-1,5){$(0,y)$}
\put(9.6,8.1){$(x,1)$}
\put(9.8,7.8){$\oplus$}
\put(9.2,5.1){$(x,y)$}
\put(6.8,5.1){$(x_2,y)$}
\put(8,7.8){$\ominus$}

\put(8,5){$\oplus$}

\put(3.95,4.75){$\oplus$}
\put(4.2,5){$(x_1,y)$}
\put(12,5){$(1,y)$}
\put(2,0.05){$\oplus$}
\put(8,5){$\oplus$}
\put(11.8,4.8){$\oplus$}
\put(10,0.05){$\oplus$}
\put(1.6,-.3){$(x,0)$}
\put(9.7,-.3){$(x,0)$}
\put(3.9,0.05){$\ominus$}
\put(11.75,0.05){$\ominus$}
\put(3.7,-.3){$(x_1,0)$}
\put(8,-.3){$(x_2,0)$}
\put(5,8.7){$D_2<0$}
\put(5.6,8.1){$(x,1)$}
\put(6,7.8){$\oplus$}
\put(7.75,7.8){$\ominus$}
\put(6.05,3.6){$(x,y)$}
\put(7.75,3.5){$\oplus$}
\put(8.05,3.5){$(x_2,y)$}
\put(4.27,0.05){$\ominus$}
\put(5.6,-.3){$(x,0)$}
\put(5.75,0.05){$\oplus$}
\put(2.9,3.4){$(x_1,y)$}
\put(4.2,3.3){$\oplus$}
\thinlines
\put(2,0){\line(0,1){8}}
\put(6,0){\line(0,1){8}}
\put(10,0){\line(0,1){8}}
\put(0,5){\line(1,0){4.2}}
\put(8,5){\line(1,0){4}}
\put(4.2,3.5){\line(1,0){3.8}}

\end{picture}
\end{center}
\normalsize
\vskip.5cm
Consequently, the following Theorem holds true.
\begin{theorem}
Let $0<x_1<x_2<1$ and
\begin{equation}\label{53}
F(x,y)=
\begin{cases}
F_1(x,y)&\text{ if }x\in(0,x_1), \frac{\partial^2F_1(x,y)}{\partial x\partial y}>0,\\
F_2(x,y)&\text{ if }x\in(x_1,x_2),\frac{\partial^2F_2(x,y)}{\partial x\partial y}<0,\\
F_3(x,y)&\text{ if }x\in(x_2,1),\frac{\partial^2F_3(x,y)}{\partial x\partial y}>0.
\end{cases}
\end{equation}
Then the copula maximizing $\int_0^1\int_0^1 F(x,y)d\tilde{C}(x,y)$ has the form
\begin{equation}
\label{25}
C(x,y)=
\begin{cases}
\min(x,h_1(y))&\text{ if }x\in[0,x_1],\\
\max(x+h_2(y)-x_2,h_1(y))&\text{ if }x\in[x_1,x_2],\\
\min(x-x_2+h_2(y),y)&\text{ if }x\in[x_2,1],
\end{cases}
\end{equation}
where $h_1(y)=C(x_1,y)$, and $h_2(y)=C(x_2,y)$. 
\end{theorem}
As we will see below, this result implies that in an ideal situation the problem is reduced to the
determination of suitable functions $h_1$ and $h_2$.\\
\\
Before going on we need to determine $dC(x,y)$ for the special situation of \eqref{25}.
For this reason let us consider the rectangles
\begin{center}
\setlength{\unitlength}{0.9cm}
\begin{picture}(6,4)
\small
\thinlines
\put(-1.7,0.4){$(x,y)$}
\put(0.6,0.4){$(x+\dd x,y)$}
\put(-1.7,3.4){$(x,y+\dd y)$}
\put(0.6,3.4){$(x+\dd x,y+\dd y)$}
\put(-0.4,0){Figure 8.}
\put(-1,.75){\line(1,0){2.5}}
\put(1.5,.75){\line(0,1){2.5}}
\put(1.5,3.25){\line(-1,0){2.5}}
\put(-1,3.25){\line(0,-1){2.5}}
\put(-1,0.75){\line(1,1){2.5}}

\put(6,.75){\line(1,0){2.5}}
\put(6,3.25){\line(0,-1){2.5}}
\put(8.5,.75){\line(0,1){2.5}}
\put(8.5,3.25){\line(-1,0){2.5}}

\put(4.8,0.3){$(x-\dd x,y)$}
\put(7.9,0.3){$(x,y)$}
\put(4.5,3.4){$(x-\dd x,y+\dd y)$}
\put(7.9,3.4){$(x,y+\dd y)$}
\put(8.5,0.75){\line(-1,1){2.5}}
\put(6.6,-0.2){Figure 9.}

\end{picture}
\end{center}
\normalsize
and the - from the copula induced - measures which are defined by
\begin{equation}
\label{4}
\gamma_C(\dd x,\dd y)=C(x,y)+C(x+\dd x,y+\dd y)-C(x,y+\dd y)-C(x+\dd x)
\end{equation}
and
\begin{equation}
\label{7}
\gamma_C(\dd x,\dd y)=C(x-\dd x,y)+C(x,y+\dd y)-C(x-\dd x,y+\dd y)-C(x,y),
\end{equation}
where $(\dd x,\dd y)$ stands for the infinitesimal rectangles from Fig. 8 and Fig. 9.\\
We consider the three regions in Fig. 7 where the second derivative changes sign separately.
\begin{itemize}
\item[(i)] $x\in(0,x_1)$.\\
Then $x=h_1(y)$ and $C(x,y)=\min(x,h_1(y))$.
Thus by \eqref{4}
\begin{equation}
\label{5}
\gamma_C(\dd x,\dd y)=h_1(y)+(h_1(y)+h_1'(y)\dd y)-h_1(y)-h_1(y)=h_1'(y)\dd y.
\end{equation}
\item[(ii)] $x\in(x_1,x_2)$.\\
Then $x=x_2-h_2(y)+h_1(y)$ and $C(x,y)=\max(x+h_2(y)-x_2,h_1(y))$.

Let us observe that
\begin{eqnarray*}
C(x,y+\dd y)&=&\max(x+h_2(y+\dd y)-x_2,h_1(y+\dd y))\\
&=&\max(h_1(y)+h'_2(y)\dd y,h_1(y)+h_1'(y)\dd y)\\
&=&h_1(y)+h_2'(y)\dd y,
\end{eqnarray*}
since for every $(x,y)$ such that $x+h_2(y)-x_2=h_1(y)$ we have $\frac{\dd x}{\dd y}+h'_2(y)=h_1'(y)$ and $\frac{\dd x}{\dd y}<0$.\\
Similarly
\begin{eqnarray*}
C(x-\dd x,y)&=&\max(x-\dd x+h_2(y)-x_2,h_1(y))\\
&=&\max(h_1(y)-\dd x,h_1(y))=h_1(y),
\end{eqnarray*}
since $\dd x>0$. Thus from \eqref{7} we have
\begin{align}
\label{23}
\gamma_C(\dd x,\dd y)&=h_1(y)+(h_1(y)+h_2'(y)\dd y)-h_1(y)-h'_1(y)-h_1(y)\dd y\nonumber\\
&=(h_2'(y)-h_1'(y))\dd y.
\end{align}
\item[(iii)] $x\in(x_2,1)$.\\
Then $x=x_2-h_2(y)+y$ and $C(x,y)=\min(x-x_2+h_2(y),y)$.
Let us observe that
\begin{eqnarray*}
C(x+\dd x,y+\dd y)&=&\min(x+\dd x-x_2+h_2(y+\dd y),y+\dd y)\\
&=&\min(y+\dd x+h_2'(y)\dd y,y+\dd y)=y+\dd y,
\end{eqnarray*}
since for every $(x,y)$ such that $x-x_2+h_2(y)=y$ we have $\dd x+h_2'(y)\dd y=\dd y$.\\
Moreover
\begin{eqnarray*}
C(x,y+\dd y)&=&\min(x-x_2+h_2(y+\dd y),y+\dd y)\\
&=&\min(y+h_2'(y)\dd y,y+\dd y)=y+h_2'(y)\dd y,
\end{eqnarray*}
since $h_2'(y)\le1$ and
\begin{equation*}
C(x+\dd x,y)=\min(x+\dd x-x_2+h_2(y),y)=\min(y+\dd x,y)=y,
\end{equation*}
since $\dd x>0$. Therefore with \eqref{4} we arrive at
\begin{align}
\label{24}
\gamma_C(\dd x,\dd y)&=C(x,y)+C(x+\dd x,y+\dd y)-C(x,y+\dd y)-C(x+\dd x)\nonumber\\
&=y+y+\dd y-(y+h_2'(y)\dd y)-y=(1-h_2'(y))\dd y.
\end{align}
\end{itemize}
Altogether the measure $\gamma_C$ of the infinitesimal rectangles and hence $dC(x,y)$ is given by
\begin{equation}
\label{37}
\gamma_C(\dd x,\dd y)=
\begin{cases}
h_1'(y)\dd y&\text{ if }x\in[0,x_1], x=h_1(y),\\
(h_2'(y)-h_1'(y))\dd y&\text{ if }x\in[x_1,x_2], x=x_2-h_2(y)+h_1(y),\\
(1-h_2'(y))\dd y&\text{ if }x\in[x_2,1],x=x_2-h_2(y)+y.
\end{cases}
\end{equation}
Our next step is to identify situations in which $C$ is indeed a copula.
\begin{theorem}\label{41}
The function $C(x,y)$ defined by \eqref{25} is a copula if and only if
\begin{itemize}
\item[(i)] $h_1(y)$ and $h_2(y)$ are increasing;
\item[(ii)] $h_1(0)=0$, $h_2(0)=0$;
\item[(iii)] $h_1(1)=x_1$, $h_2(1)=x_2$;
\item[(iv)] $0\le h_1(y)\le h_2(y)\le y$;
\item[(v)] $0\le h_1'(y)\le h_2'(y)\le 1$.
\end{itemize}
\end{theorem}
\begin{proof}
The structure of the proof is as follows: we first prove the necessary condition by showing that if $C$ is a copula,
then properties $(i)-(v)$ are satisfied. Then we exploit these properties to show that $C$ is a copula.\\
Let $C(x,y)$ be a copula and $h_1(y)=C(x_1,y)$ and $h_2(y)=C(x_2,y)$.\\
Properties $(i)-(iii)$ are straightforward. In order to prove $(iv)$, let us consider the rectangle $[x,1]\times [0,1]$.
Since $C(x,y)$ is a copula, we have

\begin{equation}\label{38}
C(x,0)+C(1,y_1)-C(x,y_1)-C(1,0)=y_1-C(x,y_1)\ge0,
\end{equation}
and thus $y_1\ge C(x,y_1)$.\\
We proceed in a similar way to prove $(v)$. Let us consider the rectangle $[x,1]\times[y_1,y_2]$. For an arbitrary copula $C(x,y)$ we have
\begin{equation}
\label{8}
C(x,y_1)+C(1,y_2)-C(x,y_2)-C(1,y_1)=C(x,y_1)+y_2-C(x,y_2)-y_1\ge0.
\end{equation}
Then
$$y_2-y_1\ge C(x,y_2)-C(x,y_1)$$ and thus
$C'(x,y)\le1$ a.e.. This implies $h'(y)\le1$ a.e. for $h(y)=C(x,y)$ (see also \cite[Theorem 2.2.7.]{Nelsen} ).
Furthermore
$$h_1'(y)\le h_2'(y)$$
since for every $(x,y)$ such that $x+h_2(y)-x_2=h_1(y)$ with $x\in[x_1,x_2]$  we have
$\frac{\dd x}{\dd y}+h'_2(y)=h_1'(y)$ and $\frac{\dd x}{\dd y}<0$.\\
On the other hand, it follows from $(v)$ and \eqref{37} that $\gamma_C(\dd x,\dd y)$ is
nonnegative for every $(x,y)\in[0,1]^2$ and by \eqref{25} that $C(x,0)=C(0,y)=0$.
Thus $C(x,y)$ is a distribution function.
We need to show that $C(x,1)=x$ and $C(1,y)=y$ for every $(x,y)\in[0,1]^2$. Indeed we have
\begin{equation*}
C(x,1)=
\begin{cases}
\min(x,h_1(1))=\min(x,x_1)=x&\text{ if }x\in[0,x_1],\\
\max(x+h_2(1)-x_2,h_1(1))=\max(x,x_1)=x&\text{ if }x\in[x_1,x_2],\\
\min(x-x_2+h_2(1),1)=\min(x,1)=x&\text{ if }x\in[x_2,1].
\end{cases}
\end{equation*}
For $x=1$ we need
\begin{equation}
\label{39}
C(1,y)=\min(1-x_2+h_2(y),y)=y.
\end{equation}
Since $h_2(1)=x_2$, then \eqref{39} is equivalent
\begin{equation}
\label{40}
1-y\ge h_2(1)-h_2(y),
\end{equation}
which holds true, since $h_2(1)-h_2(y)=(1-y)h_2'(y^*)$ for some $y^*\in(y,1)$ and the derivative satisfies $(v)$.
\end{proof}
Theorem \ref{41} implies the following bounds on candidate functions $h_1$ and $h_2$.
\begin{equation*}
\overline{h_1}(y)=
\begin{cases}
y&\text{ if }y\in[0,x_1],\\
x_1&\text{ if }y\in[x_1,1],
\end{cases}
\quad
\underline{h_1}(y)=
\begin{cases}
0&\text{ if }y\in[0,1-x_1],\\
y-(1-x_1)&\text{ if }y\in[1-x_1,1],
\end{cases}
\end{equation*}
\begin{equation*}
\overline{h_2}(y)=
\begin{cases}
y&\text{ if }y\in[0,x_2],\\
x_2&\text{ if }y\in[x_2,1],
\end{cases}
\quad
\underline{h_2}(y)=
\begin{cases}
0&\text{ if }y\in[0,1-x_2],\\
y-(1-x_2)&\text{ if }y\in[1-x_2,1],
\end{cases}
\end{equation*}
where
\begin{equation}
\label{51}
\underline{h_1}(y)\le h_1(y)\le
\overline{h_1}(y),\quad
\underline{h_2}(y)\le h_2(y)\le
\overline{h_2}(y).
\end{equation}
Now, we return to the integral \eqref{1}.
\begin{theorem}
Let us define a function $G$ by
\begin{align}\label{29}
G:=&G(y,h_1,h_2,h_1',h_2')=\nonumber\\
&=F_1(h_1(y),y)h_1'(y)+F_2(x_2-h_2(y)+h_1(y),y)(h_2'(y)-h_1'(y))\nonumber\\
&+F_3(x_2-h_2(y)+y,y)(1-h_2'(y)).
\end{align}
If $h_1,h_2$ maximize $\int_0^1G\dd y$ and satisfy Theorem \ref{41}, then
\begin{equation}\label{57}
\max_{C(x,y)-\text{ copula}}\int_0^1\int_0^1F(x,y)\dd C(x,y)=\int_0^1G\dd y.
\end{equation}
If not, then we only have the following inequality
\begin{equation}
\label{45}
\max_{C(x,y)\in\mathcal{C}}\int_0^1\int_0^1F(x,y)\dd C(x,y)\le\int_0^1G\dd y.
\end{equation}
The class $\mathcal{C}$ is the set of all copulas of the form \eqref{25} with $h_1$, $h_2$ fulfilling the assumptions of Theorem \ref{41}.
\end{theorem}
\begin{proof}
Let $F$ be a function defined on $[0,1]^2$ such that $D_2=\frac{\partial^2F(x,y)}{\partial x\partial y}$
changes its sign as indicated in Fig.\ 4. Then the two-dimensional Riemann-Stieltjes integral of $F$ with respect
to the copula $C$ defined in (\ref{25}) is given as follows
\begin{align}\label{26}
\int_0^1\int_0^1F(x,y)&dC(x,y)=
\int_0^1F_1(h_1(y),y)h_1'(y)\dd y\ +\nonumber\\
&+\int_0^1F_2(x_2-h_2(y)+h_1(y),y)(h_2'(y)-h_1'(y))\dd y\nonumber\\
&+\int_0^1F_3(x_2-h_2(y)+y,y)(1-h_2'(y))\dd y\\
&=\int_0^1 G dy\ .
\end{align}
Since under the assumptions of Theorem \ref{41}., $C$ is indeed a copula, the representation from Theorem 3. implies optimality.
The second statement is obvious, since the class $\mathcal{C}$ is a subset, due to additional restrictions, of candidate functions $h_1$, $h_2$.
\end{proof}
\begin{remark}
Note that to compute extremes of $\int_0^1 G(y,h_1,h_2,h_1',h_2')\dd y$
we can apply calculus of variations (cf \cite[p. 33]{Weinstock}).
In particular, if $(h_1,h_2)$ are extrema for the integral  $\int_0^1G(y,h_1,h_2,h_1',h_2')\dd y$,
then $(h_1,h_2)$ satisfy the Euler-Lagrange differential equations
\begin{align}
\label{42}
&\frac{\partial G}{\partial h_1}-\frac{\dd}{\dd y}\frac{\partial G}{\partial h_1'}=0\ ,\nonumber\\
&\frac{\partial G}{\partial h_2}-\frac{\dd}{\dd y}\frac{\partial G}{\partial h_2'}=0\ .
\end{align}
The solution $(h_1,h_2)$ to \eqref{42} maximizes $\int_0^1G(y,h_1,h_2,h_1',h_2')\dd y$ if
\begin{equation}
\label{50}
\frac{\partial^2  G}{\partial h_1'\partial h_1'}\le0,\quad
\begin{vmatrix}
\frac{\partial^2  G}{\partial h_1'\partial h_1'}
&\frac{\partial^2  G}{\partial h_1'\partial h_2'}\\
\frac{\partial^2  G}{\partial h_2'\partial h_1'}
&\frac{\partial^2  G}{\partial h_2'\partial h_2'}\\
\end{vmatrix}
\le0.
\end{equation}
\end{remark}
\begin{remark}
From the optimal copula with representation \eqref{25} and the properties of $h_1$ and $h_2$ from Theorem \ref{41}.
we can derive the solution of the problem in the vocabulary of optimal couplings as well. Notice that for $x\in[0,x_1)$, $(x,y)$
is mapped to $(h_1(y),y)$. According to Theorem \ref{41}. $h_1$ is monotone increasing and admits an inverse $g_1$. For $x\in[x_1,x_2)$,
we have $(x,y)$ is mapped to $(x_2-(h_2(y)-h_1(y)),y)$, where $x_2-(h_2(y)-h_1(y))$ is monotone decreasing in $y$ with inverse function $g_2$.
Finally, for $x\in[x_2,1]$, we have $(x,y)\mapsto(x_2+y-h_2(y),y)$, with $x_2+y-h_2(y)$ increasing in $y$ and inverse $g_3$. Therefore we can identify the
optimal coupling $(U,\Gamma(U))$ for $U$ uniformly distributed on $[0,1]$ and
\begin{align*}
\Gamma(x)=\left\{\begin{array}{ll}
g_1(x),&x\in[0,x_1),\\
g_2(x),&x\in[x_1,x_2),\\
g_3(x),&x\in[x_2,1].
\end{array}
\right.
\end{align*}
\end{remark}

\section{A piecewise linear cost function}
Let
\begin{equation*}
F(x,y)=
\begin{cases}
F_1(x,y)=\frac{x}{x_1}y&x\in(0,x_1),\\
F_2(x,y)=\frac{x_2-x}{x_2-x_1}y&x\in(x_1,x_2),\\
F_3(x,y)=\frac{x-x_2}{1-x_2}y&x\in(x_2,1),
\end{cases}
\end{equation*}
with $x$-component as shown in Fig. 10.
\begin{center}
\setlength{\unitlength}{0.8cm}
\begin{picture}(6,6)
\small
\thinlines
\put(0,-0.3){$0$}
\put(6,-0.3){$1$}
\put(2,-0.3){$x_1$}
\put(4,-0.3){$x_2$}
\put(0.5,1){$\frac{x}{x_1}$}
\put(2.5,5){$\frac{x_2-x}{x_2-x_1}$}
\put(4.5,1){$\frac{x-x_2}{1-x_2}$}
\put(2.3,-.8){Figure 10.}
\put(0,0){\line(1,0){6}}
\put(6,0){\line(0,1){6}}
\put(6,6){\line(-1,0){6}}
\put(0,6){\line(0,-1){6}}
\put(2,0){\line(0,1){6}}
\put(4,0){\line(0,1){6}}
\put(0,0){\line(1,3){2}}
\put(2,6){\line(1,-3){2}}
\put(4,0){\line(1,3){2}}
\end{picture}
\end{center}
\normalsize
\vskip.4cm
Towards the construction of $G$ from \eqref{29}, we identify
\begin{align*}
&F_1(h_1(y),y)h_1'(y)=\frac{h_1}{x_1}yh_1',\\
&F_2(x_2-h_2(y)+h_1(y),y)(h_2'(y)-h_1'(y))=\frac{h_2-h_1}{x_2-x_1}y(h_2'-h_1'),\\
&F_3(x_2-h_2(y)+y,y)(1-h_2'(y))=\frac{y-h_2}{1-x_2}y(1-h_2'),
\end{align*}
such that $G$ takes the form
\begin{equation}\label{30}
G=\frac{h_1}{x_1}yh_1'+\frac{h_2-h_1}{x_2-x_1}y(h_2'-h_1')+\frac{y-h_2}{1-x_2}y(1-h_2').
\end{equation}
The associated Euler-Lagrange equations are given by
\begin{align}
\label{31}
\frac{\partial G}{\partial h_1}-\frac{\dd}{\dd y}\frac{\partial G}{\partial h_1'}
&=\frac{h_1'}{x_1}y-\frac{h_2'-h_1'}{x_2-x_1}y-\frac{h_1}{x_1}+\frac{h_2-h_1}{x_2-x_1}=0,\\
\frac{\partial G}{\partial h_2}-\frac{\dd}{\dd y}\frac{\partial G}{\partial h_2'}
\label{32}
&=\frac{h_2'-h_1'}{x_2-x_1}y-\frac{1-h_2'}{1-x_2}y-\frac{h_2-h_1}{x_2-x_1}+\frac{y-h_2}{1-x_2}+\frac{y}{1-x_2}=0.
\end{align}
Now, adding \eqref{31} and \eqref{32} and multiplying the sum by $\frac{1-x_2}{y}$ we get
\begin{equation}
\label{33}
h_1'\bigg(\frac{1-x_2}{x_1}\bigg)
+h_2'=\frac{h_1}{y}\bigg(\frac{1-x_2}{x_1}\bigg)+\frac{h_2}{y}-1.
\end{equation}
Multiplication of \eqref{31} with $(\frac{x_2-x_1}{y})$ gives
\begin{equation}
\label{34}
h_1'\bigg(\frac{x_2-x_1}{x_1}+1\bigg)
-h_2'=\frac{h_1}{y}\bigg(\frac{x_2-x_1}{x_1}+1\bigg)-\frac{h_2}{y}.
\end{equation}
Summing up \eqref{33} and \eqref{34} we find
\begin{equation}
\label{35}
h_1'=\frac{h_1}{y}-x_1
\end{equation}
and hence
\begin{equation}
\label{36}
h_2'=\frac{h_2}{y}-x_2.
\end{equation}
The general solution $h(y)$ of the differential equation
\begin{equation}
\label{43}
h'=\frac{h}{y}-x,
\end{equation}
has the form
\begin{equation}
\label{44}
h(y)=c\,y-xy\,\log y.
\end{equation}
From the boundary conditions $h(1)=x$ and $h(0)=0$, we find $h(y)=xy(1-\log y)$ and $h'(y)=x(-\log y)$.
Thus
\begin{equation}
\label{52}
h_1(y)=x_1y(1-\log y), \quad h_2(y)=x_2y(1-\log y),
\end{equation}
which unfortunately do not satisfy condition $(v)$ in Theorem \ref{41}. Finally, $G$ is given through
\begin{equation*}
G=y^2\log y\bigg(\frac{x_2}{1-x_2}\bigg)+y^2(\log y)^2\bigg(\frac{x_2}{1-x_2}\bigg)+y^2
\end{equation*}
which yields the value
\begin{equation}
\int_0^1G\dd y=\frac{x_2}{1-x_2}\bigg(-\frac{1}{27}\bigg)+\frac{1}{3}.
\end{equation}
On the other hand, if $h_1(y)=x_1y$ and $h_2(y)=x_2y$, then $G=y^2$
and $\int_0^1G\dd y=\frac{1}{3}$. Thus \eqref{52} does not maximize $\int_0^1G\dd y$.\\
Note that for $F(x,y)=f(x)y$, with $f(x)$ uniform distribution preserving map (u.d.p.) we have
\begin{align}
&\max_{x_n,y_n\text{ are u.d.}}\lim_{N\to\infty}\frac{1}{N}\sum_{n=1}^N
f(x_n)y_n=\max_{C(x,y)-\text{copula}}\int_0^1\int_0^1F(x,y)d C(x,y),\label{47}\\
&\max_{x_n-\text{u.d}, \Phi-\text{u.d.p.}}\lim_{N\to\infty}\frac{1}{N}\sum_{n=1}^N f(x_n)\Phi(x_n)
=\max_{\Phi-\text{u.d.p.}}\int_0^1 f(x)\Phi(x)\dd x.\label{48}
\end{align}
Then \eqref{48} is in general smaller than \eqref{47}. By \cite[Corollary 3]{FialovaStrauch}
for u.d.p. $f(x)$ we have
\begin{equation}
\max_{\Phi-\text{u.d.p.}}\int_0^1\int_0^1f(x)\Phi(x)\dd x=\frac{1}{3}.
\end{equation}
Therefore in our situation we have $\int_0^1\int_0^1F(x,y)dC(x,y)\ge\frac{1}{3}$.
\begin{remark}
The example points out the deficiencies of the variational formulation in the present context.
When maximizing \eqref{1} it is essential to preserve the uniform distribution property of the marginals.
In the formulation via the function \eqref{29}, which led to the problem from the calculus of variations, this constraint is not present anymore
and the maximization takes place over differentiable $h_1,\,h_2$.
On the other hand the optimal copula $C$ in \eqref{25} with $h_1(y)=C(x_1,y)$ and $h_2(y)=C(x_2,y)$ does not enforce any smoothness properties,
which implies that when solving \eqref{29} one does not necessarily get an upper bound for \eqref{1}. The same reasoning suggests that
when maximizing \eqref{29} over differentiable functions, which fulfill the conditions stated in Theorem \ref{41}, one in general derives
a lower bound for \eqref{1}.
\end{remark}

\section{A different approach using coupling}
In this section we consider the example
\begin{equation*}
F(x,y)=\sin(\pi(x+y))
\end{equation*}
and relate this problem to combinatorial optimization.
In \cite{HoferIaco} an upper bound for
\begin{equation}\label{sine}
\int_{[0,1]^2}\sin(\pi(x+y))\gamma(dx,dy)\ .
\end{equation}
was found by means of the Hungarian Algorithm and we will show that the copula found in \cite{HoferIaco} is indeed the one maximizing \eqref{sine}.

For dealing with this particular example one may utilize Theorem 1. from \cite{Uckelmann}.
But since the proof of what we actually need is not given there, we re-state the following particular version of this
result and give its proof.
\begin{theorem}
Let $\mu,\,\nu$ be the uniform distribution on $[0,1]$ and the cost function $c(x,y)=\phi(x+y)$ with
$\phi:[0,2]\to\mathbb{R}$. In particular we assume that $\phi\in\mathcal{C}^2[0,1]$ and that there is $k\in(0,2)$
such that $\phi''(x)<0$ for $x\in[0,k)$ and $\phi''(x)>0$ for $x\in(k,2]$. If $\beta\in(0,1)$ denotes the solution to
$$\phi(2\beta)-\phi(\beta)=\beta\phi'(\beta),$$
then
\begin{align*}
\Gamma(x)=
\left\{\begin{array}{cc}
\beta-x,& x\in[0,\beta),\\
x,&x\in[\beta,1],
\end{array}\right.
\end{align*}
induces by $(U,\Gamma(U))$ for some standard uniformly distributed $U$ an optimal $c$-coupling between $P$ and $Q$.
\end{theorem}
\begin{proof}
For the proof we proceed as proposed in \cite{Uckelmann} and \cite{RueschUckel}.
Define the following functions:
\begin{align*}
f_1(x)&=x\phi'(\beta),\\
f_2(x)&=\frac{1}{2}(\phi(2 x)-\phi(2\beta))+\beta \phi'(\beta),\\
\psi^1(\xi)&=\phi(\beta-x+\xi)+x\phi'(\beta)-\phi(\beta),\\
\psi^2(\xi)&=\phi(x+\xi)-\frac{1}{2}\phi(2x)-\frac{1}{2}\phi(2\beta)+\beta\phi'(\beta).
\end{align*}
Furthermore set
$$f(x)=f_1(x)I_{[0,\beta)}(x)+f_2(x) I_{[\beta,1]}(x),$$
and put for $\xi\in[0,1]$:
\begin{align*}
\psi_{\Gamma(x)}(\xi)=
\left\{\begin{array}{cc}
\psi^1(\xi),& x\in[0,\beta),\\
\psi^2(\xi),& x\in[\beta,1].
\end{array}\right.
\end{align*}
Here $\psi_{\Gamma(x)}(\xi)$ plays the role of $\psi_{y,a}(\xi)=c(\xi,y)+a(y)$ with $y=\Gamma(x)$ in (\ref{eq:ubdiffeasy}).
Now the idea, following Theorem \ref{Th:optimality} and (\ref{eq:ubdiffeasy}),
is to show that $y=\Gamma(x)$ is in the $c$-subdifferential of $f(x)$ for all $x\in[0,1]$ which implies optimality
of this particular coupling and optimality of the distribution induced by $(U,\Gamma(U))$ for the transport problem.
For the $c$-convexity of $f$ and the subdifferential property we need to show:
\begin{align*}
\psi_{\Gamma(x)}(x)&=f(x)\quad\forall\,x\in[0,1],\\
\psi_{\Gamma(x)}(\xi)&\leq f(\xi)\quad\forall\,\xi\in[0,1].
\end{align*}
We start with showing that $\psi_{\Gamma(x)}(x)=f(x)$. For $x\in[0,\beta)$ we have that $\Gamma(x)=\beta-x$ and
$$\psi_{\Gamma(x)}(x)=\psi^1(x)=x\phi'(\beta)=f_1(x)=f(x).$$
For $x\in[\beta,1]$ we have $\Gamma(x)=x$ and
$$\psi_{\Gamma(x)}(x)=\psi^2(x)=\frac{1}{2}(\phi(2 x)-\phi(2\beta))+\beta\phi'(\beta)=f_2(x)=f(x).$$
It remains to show $\psi_{\Gamma(x)}(\xi)\leq f(\xi)$ for all $(x,\xi)\in[0,1]\times [0,1]$.\\
At first we need some details on the location of $\beta$. From the definition of $\beta$ and the
mean value theorem we obtain $\exists \tau\in(\beta,2\beta)$ with $\phi'(\tau)=\phi'(\beta)$.
Since $\phi$ is concave on $[0,k)$ and convex on $(k,2]$ we see that $\beta<k<\tau<2\beta$.\\
In the following we distinguish four situations.
\begin{itemize}
\item For $x\in[0,\beta)$ and $\xi\in[0,\beta)$ the functions are given by
\begin{align*}
\psi_{\Gamma(x)}(\xi)=\phi(\beta-x+\xi)+x\phi'(\beta)-\phi(\beta),\quad f(\xi)=\xi\phi'(\beta).
\end{align*}
Therefore we need to check:
\begin{align}\label{eq:1st}
f(\xi)-\psi_{\Gamma(x)}(\xi)=(\xi-x)\phi'(\beta)-(\phi(\beta+\xi-x)-\phi(\beta))=:F_1(x,\xi).
\end{align}
Obviously $F_1(x,\xi)=0$ for $(0,\beta)$, $(0,0)$ and $(x,x)$. In general we have
by concavity on $[0,k)$ and the definition of $\beta$, $\phi(2\beta)=\phi(\beta)+\beta\phi'(\beta)$
\begin{align*}
\phi(\beta)+(\beta+\xi-x-\beta)\phi'(\beta)\geq\phi(\beta+\xi-x),
\end{align*}
since $\beta+\xi-x\in[0,2\beta]$, which proves that (\ref{eq:1st}) is positive.
\item For $x\in[0,\beta)$ and $\xi\in[\beta,1]$ the functions are given by
\begin{align*}
&\psi_{\Gamma(x)}(\xi)=\phi(\beta-x+\xi)+x\phi'(\beta)-\phi(\beta)\ ,\\
&f(\xi)=\frac{1}{2}(\phi(2\xi)-\phi(2\beta))+\beta\phi'(\beta).
\end{align*}
We need the following to be positive:
\begin{align*}
&\frac{1}{2}(\phi(2\xi)-\phi(2\beta))+\beta\phi'(\beta)-\phi(\beta+\xi-x)-x\phi'(\beta)+\phi(\beta)\\
&=\frac{1}{2}\phi(2\xi)+\frac{1}{2}\phi(2\beta)-x\phi'(\beta)-\phi(\beta+\xi-x)\\
&\geq \phi(\xi+\beta)-x\phi'(\beta)-\phi(\beta+\xi-x)=:F_2(x,\xi),
\end{align*}
where the equality above follows from the definition of $\beta$ and the inequality follows from convexity since $k<2\beta\leq 2\xi$.
Observe $F_2(x,\xi)=0$ for $(0,\xi)$ and $(\beta,\beta)$. To prove positivity the strategy is as follows, firstly show
$F_2(x,\xi=\beta)\geq0$ for all $x\in[0,\beta)$ and secondly show that $\frac{\partial}{\partial\xi}F_2(x,\xi)\geq0$
for all $(x,\xi)\in[0,\beta)\times(\beta,1)$.\\
Look at
\begin{align*}
\frac{\partial}{\partial x}F_2(x,\xi=\beta)=-\phi'(\beta)+\phi'(2\beta-x)
\end{align*}
which is zero in $(0,\beta)$ exactly if $x=\hat{x}=2\beta-\tau<\beta$.
Since $0<\phi''(\tau)=-\frac{\partial^2}{\partial x^2}F_2(x,\xi=\beta)\vert_{x=\hat{x}}$ we have that
$$F_2(\hat{x},\xi=\beta)=\phi(2\beta)-(2\beta-\tau)\phi'(\tau)-\phi(\tau)>0$$ is a maximum, it is positive by convexity ($k<\tau<2\beta$).
Thus, $F_2(x,\xi=\beta)\geq 0$ for $x\in[0,\beta)$.
Now we can deal with checking the interior,
\begin{align*}
\frac{\partial}{\partial \xi}F_2(x,\xi)=\phi'(\beta+\xi)-\phi'(\beta+\xi-x).
\end{align*}
Suppose $\beta+\xi-x\geq k$, then by convexity $\phi'(\beta+\xi)\geq\phi'(\beta+\xi-x)$. On the other hand if
$\beta+\xi-x< k$ we have $\phi'(\beta+\xi-x)<\phi'(\beta)=\phi'(\tau)\leq\phi'(2\beta)\leq\phi'(\beta+\xi)$ since $2\beta\leq\beta+\xi$.
\item Consider $x\in[\beta,1]$ and $\xi\in[0,\beta)$, here
\begin{align*}
\psi_{\Gamma(x)}(\xi)=\phi(x+\xi)-\frac{1}{2}\phi(2 x)-\frac{1}{2}\phi(2\beta)+\beta\phi'(\beta),\quad f(\xi)=f_1(\xi)=\xi\phi'(\beta),
\end{align*}
and we need
\begin{align*}
&\xi\phi'(\beta)-\phi(x+\xi)+\frac{1}{2}\phi(2 x)+\frac{1}{2}\phi(2\beta)-\beta\phi'(\beta)\\
&\geq \phi(x+\beta)-\phi(x+\xi)-(\beta-\xi)\phi'(\beta)=:F_3(x,\xi)
\end{align*}
to be positive, the inequality stems from convexity since $2x\geq 2\beta>k$.
We proceed as before. $F_3(x,\xi)=0$ for $(x,\beta)$ and $(\beta,0)$. At first fix $x=\beta$,
$F_3(x=\beta,y)=0$ if $y=0$ and $y=\beta$. In between we study
\begin{align*}
\frac{\partial}{\partial\xi}F_3(x=\beta,\xi)=-\phi'(\beta+\xi)+\phi'(\beta),
\end{align*}
which is zero in $(0,\beta)$ exactly if $\xi=\hat{\xi}=\tau-\beta$.
Again due to convexity of $\phi$ on $(k,2]$ we have a maximum in $\hat{\xi}$ and
$$F_3(x=\beta,\hat{\xi})=\phi(2\beta)-\phi(\tau)-(2\beta-\tau)\phi'(\tau)>0,$$
such that we have $F_3(x=\beta,\xi)\geq 0$ for $\xi\in[0,\beta)$. On the interior we show that
\begin{align*}
\frac{\partial}{\partial x} F_3(x,\xi)=\phi'(\beta+x)-\phi'(x+\xi)\geq 0.
\end{align*}
If $x+\xi\geq k$ we have from convexity $\phi'(\beta+x)\geq \phi'(x+\xi)$.
If $x+\xi < k$ we have $\beta\leq x+\xi<k\leq \tau<2\beta\leq\beta+x$ and therefore
$$\phi'(x+\xi)\leq\phi'(\beta)=\phi'(\tau)\leq \phi'(2\beta)\leq \phi'(\beta+x).$$
\item Let $x\in[\beta,1]$ and $\xi\in[\beta,1]$. Here
\begin{align*}
&\psi_{\Gamma(x)}(\xi)=\phi(x+\xi)-\frac{1}{2}\phi(2 x)-\frac{1}{2}\phi(2\beta)+\beta\phi'(\beta)\ ,\\
& f(\xi)=\frac{1}{2}
(\phi(2\xi)-\phi(2\beta))+\beta\phi'(\beta),
\end{align*}
such that
\begin{align*}
f(\xi)-\psi_{\Gamma(x)}(\xi)=\frac{1}{2}\phi(2x)+\frac{1}{2}\phi(2\xi)-\phi(x+\xi)\geq0,
\end{align*}
is fulfilled by convexity since $x+\xi,\,2x,\,2\xi\geq 2\beta>k$.
\end{itemize}
We can conclude that $\psi_{\Gamma(x)}(\xi)\leq f(\xi)\quad\forall\,(x,\xi)\in[0,1]\times[0,1]$, which according
to Theorem \ref{Th:optimality}
shows that the vector $(U,\Gamma(U))$ for $U\stackrel{d}{\sim}\mathcal{U}([0,1])$ yields an optimal coupling.
\end{proof}
\begin{remark}
If $\beta>1$ then it can be shown as in the first step of the above proof that $(U,1-U)$ yields the optimal coupling.
Loosely speaking one could say that the concave behaviour dominates the convex one.
\end{remark}
Now we are prepared to answer the sine question. Setting $\phi(z)=\sin(\pi z)$ and $k=1$ we immediately get:
\begin{corollary}
For $c(x,y)=\sin(\pi(x+y))$ we have that the distribution of the vector $(U,\Gamma(U))$ for $U\sim\mathcal{U}([0,1])$ with
\begin{align*}
\Gamma(x)=\left\{\begin{array}{cc}
\beta-x,& x\in[0,\beta),\\
x,&x\in[\beta,1],
\end{array}\right.
\end{align*}
and $\beta=0.7541996008265638\approx0.7542$ (cf. \cite{HoferIaco}) which solves
\begin{align}\label{eq:1storder}
\sin(2\pi\beta)-\sin(\pi\beta)=\beta\pi\cos(\pi\beta),
\end{align}
is maximizing
$$\int_{[0,1]^2} \sin(\pi(x+y)) dC(x,y)$$
in the set of all bivariate distributions $C$ with uniform marginals, i.e., in the set of all copulas.
\end{corollary}
\begin{remark}
In this situation equation (\ref{eq:1storder}) meets the first order condition when looking at couplings of the form
$(U,\Gamma^{\alpha}(U))$ with
\begin{align*}
\Gamma^{\alpha}(x)=\left\{
\begin{array}{cc}
\alpha-x,& x\in[0,\alpha),\\
x,& x\in[\alpha,1],
\end{array}\right.
\end{align*}
or explicitly maximizing ($c(x,y)=\sin(\pi(x+y)$)
\begin{align*}
H(\alpha):=\int_0^{\alpha }c(x,\alpha -x) \, dx+\int_{\alpha }^1 c(x,x) \, dx.
\end{align*}
This is conjectured in \cite{HoferIaco} but with inaccurate $\alpha=\frac{3}{4}$.
\end{remark}

\section*{Acknowledgments}
M.R. Iac\`o and R.F. Tichy would like to thank the Slovak colleagues for their hospitality during the visit in Bratislava in November 2014.
{\footnotesize
\bibliographystyle{abbrv}
\bibliography{strauch}}

\end{document}